\documentclass{birkjour}
\usepackage{amsmath,amssymb,mathtools,enumitem}
\usepackage[hidelinks]{hyperref}
\numberwithin{equation}{section}
\renewcommand{\subjclassname}{Mathematics Subject Classification (2000)}
\newtheorem{theorem}{Theorem}[section]
\newtheorem{lemma}[theorem]{Lemma}
\newtheorem{proposition}[theorem]{Proposition}
\newtheorem{corollary}[theorem]{Corollary}
\theoremstyle{definition}

\newtheorem{example}[theorem]{Example}
\newtheorem{remark}[theorem]{Remark}
\newtheorem*{maintheorem}{Main Theorem}
\newcommand{\N}{\mathbb N}\newcommand{\Z}{\mathbb Z}
\newcommand{\KP}{\operatorname{KP}}\newcommand{\Spec}{\operatorname{Spec}}
\newcommand{\Prim}{\operatorname{Prim}}\newcommand{\MT}{\operatorname{MT}}
\newcommand{\GSpec}{\operatorname{GSpec}}\newcommand{\Ind}{\operatorname{Ind}}
\newcommand{\Ann}{\operatorname{Ann}}\newcommand{\MaxSpec}{\operatorname{MaxSpec}}
\newcommand{\Frac}{\operatorname{Frac}}

\begin{document}
\title[Locally Closed Primes and Isotropy Induction]
{Locally Closed Prime Ideals and Isotropy Induction in Kumjian--Pask Algebras over Commutative Rings}
\author[Nguyen Bich Van]{Nguyen Bich Van}
\address{Institute for Artificial Intelligence\\
VNU University of Engineering and Technology\\
Hanoi, Vietnam}
\email{nbvan@vnu.edu.vn}
\begin{abstract}
Let $R$ be a nonzero commutative unital ring and let $\Gamma$ be a
row-finite, strongly aperiodic higher-rank graph with no sources and a
nonempty finite vertex set. We determine the prime ideals that occur as
annihilators of modules induced from simple isotropy modules. They are
parametrised by maximal tails of $\Gamma$ and prime ideals $\mathfrak p$
of $R$ for which the fraction field of $R/\mathfrak p$ is obtained by
inverting a single element. Under the product description of the prime
spectrum, these are precisely its locally closed points. The realisation
argument uses a dense periodic orbit in every finite maximal-tail quotient.
We also construct an explicit faithful simple module for the Leavitt
algebra $L_R(1,2)$ whenever $R$ is a Goldman domain. Consequently,
$L_{k[t]_{(t)}}(1,2)$ is primitive although its coefficient ring is not a
field. In this example the zero ideal is locally closed but not completely
irreducible. We obtain full prime-to-primitive coincidence for
coefficient rings with finite spectrum and for semilocal rings of
Krull dimension at most one, and transfer the results to split
pullback Kumjian--Pask algebras. For pullbacks over an uncountable field,
the full primitive spectrum is described without assuming algebraic
closedness.

\end{abstract}
\maketitle

\section{Introduction}
The passage from field coefficients to commutative ring coefficients
changes the relation between prime ideals and simple representations of
graph algebras. Coefficient ideals already enter the ideal structure of
Leavitt path algebras \cite{Tomforde} and Kumjian--Pask algebras
\cite{APCHR}. For strongly effective groupoids, Clark, Edie-Michell,
an Huef and Sims described arbitrary ideals by coefficient data
\cite{CEHS}. Rigby and van den Hove subsequently developed an ideal
classification for general Steinberg algebras over arbitrary commutative
rings \cite{RVH}. These results provide the ideal-theoretic foundation
for the present work.

Our concern is the additional restriction imposed by \emph{simple}
isotropy induction. The isotropy group of an infinite path in a
higher-rank graph is a subgroup of a finitely generated free abelian
group. Its group algebra is therefore a commutative algebra of finite
type over the coefficient ring. A simple isotropy module is a field
quotient of this algebra. The image of the coefficient ring in such a
field need not itself be a field: its fraction field need only be
generated by the inverse of one element. This is the Goldman condition
\cite{Goldman}. It is also the condition for a prime of a commutative
ring to be a locally closed point of its spectrum.

For a nonzero commutative unital ring $R$, write
\[
 \GSpec(R)=\{\mathfrak p\in\Spec R:
 \Frac(R/\mathfrak p)=(R/\mathfrak p)[s^{-1}]
 \text{ for some }0\ne s\in R/\mathfrak p\}.
\]
Write $\Prim_{\mathrm{ind}}(A)$ for the annihilators obtained by inducing
simple isotropy modules of the specified graph groupoid. The
subscript records the source of the representations; it does not mean
the whole primitive spectrum.

\begin{maintheorem}
Let $\Gamma$ be a row-finite, strongly aperiodic $\ell$-graph with no
sources, $\ell\geq1$, and $0<|\Gamma^0|<\infty$. Set $A=\KP_R(\Gamma)$.
For a maximal tail $T$, put $H_T=\Gamma^0\setminus T$. Then:
\begin{enumerate}[label=\textup{(\roman*)},leftmargin=2em]
\item The assignment $(T,\mathfrak p)\mapsto I_{H_T}+\mathfrak pA$
is a homeomorphism $\MT(\Gamma)\times\Spec R\longrightarrow\Spec A$.
\item Under this homeomorphism,
\[
 \Prim_{\mathrm{ind}}(A)
 \cong\MT(\Gamma)\times\GSpec(R).
\]
These are exactly the locally closed points of $\Spec A$.
\item Every ideal in \textup{(ii)} is realised using a periodic path
whose orbit is dense in the corresponding maximal-tail quotient.
\end{enumerate}
\end{maintheorem}

Part (i) is a consequence of the established uniqueness and
coefficient-ideal theorems. We include its proof to specify both the
parameter space and its topology. The representation-theoretic step is
the converse in (ii): every Goldman prime is realised. Finiteness of
the vertex set provides, in each tail, a loop whose degree is positive
in every coordinate and which is reached from every surviving vertex.
Its periodic infinite repetition supplies nontrivial isotropy. One
isotropy generator can then act by $s^{-1}$, producing the required
field quotient of the isotropy algebra. The uniqueness theorem
identifies the induced annihilator.

Local closedness has already been studied for Leavitt path algebras
over fields. Abrams, Bell and Rangaswamy established the
Dixmier--Moeglin equivalence for finite graphs \cite{ABR}.
Rangaswamy characterised locally closed prime ideals of arbitrary
Leavitt path algebras over a field in terms of strong primitivity and
complete irreducibility \cite{Rangaswamy}. Ko\c{c} and \"Ozayd\i n
study simple modules over commutative coefficient rings for finite
graphs whose cycles are pairwise disjoint \cite{KocSimple}.
Their related preprint \cite[Theorem~32(ii)]{KocClassification}
parametrises the cycle-supported simple modules by maximal ideals of
$R[x,x^{-1}]$. This parametrisation already allows nonmaximal
contractions to $R$. Indeed, if $\mathfrak p$ is Goldman and
$\Frac(R/\mathfrak p)=(R/\mathfrak p)[\bar s^{-1}]$, choose a lift
$s\in R$ of $\bar s$. Then
$\mathfrak pR[x,x^{-1}]+(sx-1)$ is a maximal ideal of
$R[x,x^{-1}]$ with contraction $\mathfrak p$.
For example, $(tx-1)\subset k[t]_{(t)}[x,x^{-1}]$ contracts to zero.
Thus this arithmetic mechanism is already present there.

Our contribution is the identification of the simple-isotropy-induced
annihilators with the locally closed prime ideals for finite strongly
aperiodic higher-rank graphs over arbitrary commutative coefficients.
The two-loop rose used below is outside the pairwise-disjoint-cycle
class. In our setting, a dense periodic orbit realises each Goldman
prime, and Cuntz--Krieger uniqueness forces its induced annihilator
to be exactly $I_{H_T}+\mathfrak pA$. We classify this set of
annihilators, rather than the isomorphism classes of all simple modules.
We do not claim a new field-coefficient Dixmier--Moeglin theorem.
Indeed, the example $L_{k[t]_{(t)}}(1,2)$ separates local closedness from
complete irreducibility, as shown in
Example~\ref{ex:not-completely-irreducible}.

The same example shows why a coefficient-field condition cannot be
necessary for primitivity over general commutative rings. A direct
module construction is given in Proposition~\ref{prop:explicit-goldman-module}.
Its orbit basis and weighting are closely related to Chen's rational
path modules and scaling twists over fields \cite[Section~6]{Chen}.
The point here is that the weighted module remains simple when the
coefficient action is restricted to a Goldman domain.
It contradicts the coefficient-field necessity assertions in
\cite[Theorem~6.7, arXiv version~2]{Larki} and
\cite[Theorem~3.7, arXiv version~2]{KLA}. This comparison is with the
specified accessible versions. The argument here is independent of
those necessity assertions.

It is essential to distinguish simple isotropy induction from general
induction of primitive ideals. Steinberg's algebraic Effros--Hahn
theorem \cite{SteinbergEH} does not assert in full generality that the
inducing isotropy module is simple. For arbitrary coefficient rings,
\cite[Theorem~4.2]{DGS} proves that induction preserves simplicity.
Its exhaustivity theorem requires the isotropy rings to be left max,
meaning that every nonzero left module has a maximal submodule
\cite[Theorem~4.3]{DGS}. That additional hypothesis is not assumed here. The main theorem identifies the simple-induced locus;
we obtain the full primitive spectrum only under the additional
coefficient hypotheses stated below.

The paper is organised as follows. Section~\ref{sec:preliminaries}
fixes the graph, ideal and induction conventions.
Sections~\ref{sec:prime-spectrum} and~\ref{sec:goldman-core} prove the
main theorem. Section~\ref{sec:nonfield-examples} gives the direct
nonfield representation and its consequences.
Section~\ref{sec:pullbacks} treats split pullbacks and the
uncountable-field primitive spectrum.

\section{Preliminaries}
\label{sec:preliminaries}
Coefficient rings are commutative, nonzero and unital. Modules over
unital algebras are unital; modules over algebras with local units are
unitary. Ideals are two-sided unless stated otherwise. For a unital
ring $B$, the prime spectrum carries the hull--kernel topology, with
closed sets $V_B(J)=\{P\in\Spec B:J\subseteq P\}$. We write
$D_B(b)=\{P:b\notin P\}$ and $D_B(J)=\Spec B\setminus V_B(J)$.
The primitive spectrum is the subspace consisting of annihilators of
simple left modules.

\subsection{Higher-rank graphs and maximal tails}
A $k$-graph is a countable category $\Lambda$ with a functor
$d:\Lambda\to\N^k$ satisfying unique factorisation: if
$d(\lambda)=m+n$, then there are unique composable paths $\mu,\nu$
with degrees $m,n$ and $\lambda=\mu\nu$. We use the convention that
$v\Lambda w$ consists of paths from $v$ to $w$, with range $v$ and
source $w$. Row-finiteness and absence of sources mean
$0<|v\Lambda^n|<\infty$ for every vertex $v$ and $n\in\N^k$.
An infinite path is a degree-preserving functor from the usual path
category $\Omega_k$ to $\Lambda$. Its shift by $m$ is denoted
$\sigma^m x$.

A subset $H\subseteq\Lambda^0$ is hereditary if all descendants of its
vertices belong to $H$, and saturated if
$s(v\Lambda^n)\subseteq H$ for some $n\in\N^k$ implies $v\in H$.
For saturated hereditary $H$, the quotient $\Lambda\setminus H$
consists of the paths whose source lies outside $H$. A graph is
aperiodic if each vertex admits an infinite path $x$ for which
$\sigma^m x=\sigma^n x$ implies $m=n$. It is strongly aperiodic if
every nonempty saturated-hereditary quotient is aperiodic.

A nonempty $T\subseteq\Lambda^0$ is a maximal tail if:
\begin{enumerate}[label=\textup{(MT\arabic*)},leftmargin=3.6em]
\item every ancestor of a vertex of $T$ belongs to $T$;
\item for every $v\in T$ and $n\in\N^k$ there is
$\lambda\in v\Lambda^n$ with $s(\lambda)\in T$;
\item any two vertices of $T$ have a common descendant in $T$.
\end{enumerate}
Its complement $H_T$ is saturated hereditary. The sets
$\mathcal S(v)=\{T\in\MT(\Lambda):v\in T\}$ form a basis for the
maximal-tail topology. Indeed, a tail containing $v,w$ contains a
common descendant $z$, and every tail containing $z$ contains $v,w$.
When the vertex set is finite, $\MT(\Lambda)$ is a finite $T_0$ space.

\subsection{Kumjian--Pask algebras and their ideals}
The $R$-algebra $\KP_R(\Lambda)$ is generated by mutually orthogonal
vertex idempotents $p_v$, real paths $s_\lambda$ and ghost paths
$s_{\lambda^*}$, subject to the vertex and composition relations and
\[
 s_{\lambda^*}s_\mu=\delta_{\lambda,\mu}p_{s(\lambda)}
 \quad(d(\lambda)=d(\mu)),\qquad
 p_v=\sum_{\lambda\in v\Lambda^n}s_\lambda s_{\lambda^*}.
\]
We regard vertices as degree-zero paths. The algebra is spanned by
$s_\alpha s_{\beta^*}$ with $s(\alpha)=s(\beta)$, and
$rp_v\ne0$ whenever $r\ne0$. The degree of this monomial is
$d(\alpha)-d(\beta)$. If $\Lambda^0$ is finite, its identity is
$\sum_v p_v$. For saturated hereditary $H$, let $I_H$ be the ideal
generated by its vertices.

We use the following established facts.
\begin{enumerate}[leftmargin=2em]
\item The graded uniqueness theorem applies to graded homomorphisms
which do not kill nonzero scalar multiples of vertices
\cite[Theorem~4.1]{APCHR}.
\item If $\Lambda$ is aperiodic, the same coefficient-vertex condition
implies injectivity without a grading assumption
\cite[Theorem~4.7]{APCHR}.
\item Quotient graphs and change of coefficients give
\begin{equation}\label{eq:quotient-coefficients}
 \KP_R(\Lambda)/(I_H+J\KP_R(\Lambda))
 \cong\KP_{R/J}(\Lambda\setminus H).
\end{equation}
\item If $\Lambda$ is strongly aperiodic, every ideal is generated by
the scalar vertices it contains \cite[Theorem~6.3]{CEHS}.
\end{enumerate}
The coefficient-change isomorphism in
\eqref{eq:quotient-coefficients} follows directly from the universal
relations after imposing $J=0$; the quotient-graph identification is
the standard saturated-hereditary quotient construction.

\subsection{The graph groupoid and induction}
The graph groupoid has unit space $\Lambda^\infty$ and arrows
\[
 \mathcal G_\Lambda=
 \{(x,m-n,y):\sigma^m x=\sigma^n y,\ m,n\in\N^k\}.
\]
Range and source are $x,y$, respectively, and the middle coordinates
add under composition. The compact open bisections
\[
 Z(\alpha,\beta)=
 \{(\alpha z,d(\alpha)-d(\beta),\beta z):
                          z\in s(\alpha)\Lambda^\infty\}
\]
form a basis when $s(\alpha)=s(\beta)$. Under the standard
identification of $\KP_R(\Lambda)$ with its Steinberg algebra, the
monomial $s_\alpha s_{\beta^*}$ corresponds to
$1_{Z(\alpha,\beta)}$.

For a unit $x$, put $G_x=\mathcal G_{\Lambda,x}^x$ and
$L_x=\{\gamma\in\mathcal G_\Lambda:s(\gamma)=x\}$.
The free $R$-module $RL_x$ is a right $R[G_x]$-module by composition
and a left Steinberg-algebra module by convolution. For a left
$R[G_x]$-module $V$, define
\begin{equation}\label{eq:induction}
 \Ind_x(V)=RL_x\otimes_{R[G_x]}V.
\end{equation}
Choosing one arrow from $x$ to each point of its orbit identifies
\eqref{eq:induction}, as an $R$-module, with a direct sum of copies of
$V$ over that orbit. We include the simplicity argument over arbitrary
coefficients in Lemma~\ref{lem:simple-induction-rings}; the construction
is the usual groupoid induction of \cite{Steinberg2010}.

\section{The prime spectrum over a commutative coefficient ring}
\label{sec:prime-spectrum}

Throughout this paper, $R$ is a nonzero commutative ring with identity and
$\Gamma$ is a row-finite, strongly aperiodic
$\ell$-graph with no sources, a nonempty finite vertex set, and $\ell\geq1$.
Set $A=\operatorname{KP}_R(\Gamma)$. For a maximal tail $T$ put
$H_T=\Gamma^0\setminus T$, $\Gamma_T=\Gamma\setminus H_T$, and
\[
 P_{T,\mathfrak p}=I_{H_T}+\mathfrak p A.
\]
Here $I_{H_T}$ is generated by the vertices outside $T$.
\begin{proposition}[Prime spectrum]
\label{prop:all-prime-coefficients}
The map
\[
 \operatorname{MT}(\Gamma)\times\operatorname{Spec}(R)
 \longrightarrow\operatorname{Spec}(A),\qquad
 (T,\mathfrak p)\longmapsto P_{T,\mathfrak p},
\]
is a homeomorphism for the maximal-tail topology and the Zariski topology.
The corresponding quotient is
\[
 A/P_{T,\mathfrak p}\cong
 \operatorname{KP}_{R/\mathfrak p}(\Gamma_T).
\]
\end{proposition}
\begin{proof}
First let $P$ be prime. Put $H=\{v:p_v\in P\}$ and
$\mathfrak p=\{r\in R:r1_A\in P\}$. The set $H$ is saturated hereditary
by the path relations and the Cuntz--Krieger sums. Since the vertices
sum to $1_A$, $H\ne\Gamma^0$. The map $R\to Z(A/P)$ has kernel
$\mathfrak p$, and the center of a prime ring is a domain; hence
$\mathfrak p$ is prime. A nonzero central element $c$ of a prime unital
ring is regular: $cb=0$ implies $c(A/P)b=0$, and primeness then gives
$b=0$. Thus, for $v\notin H$,
\[
 rp_v\in P\quad\Longleftrightarrow\quad r\in\mathfrak p.
\]
The natural map
$\operatorname{KP}_{R/\mathfrak p}(\Gamma\setminus H)\to A/P$
therefore satisfies the coefficient-vertex hypothesis of
Cuntz--Krieger uniqueness \cite[Theorem~4.7]{APCHR}.
Strong aperiodicity gives injectivity, whence
$P=I_H+\mathfrak pA$. Primeness forces (MT3) on $T=\Gamma^0\setminus H$:
if two surviving vertices $v,w$ have no common descendant, then
$p_v\operatorname{KP}_{R/\mathfrak p}(\Gamma\setminus H)p_w=0$.
Indeed this corner is spanned by the monomials
$s_\alpha s_{\beta^*}$ with $r(\alpha)=v$, $r(\beta)=w$ and
$s(\alpha)=s(\beta)$, and there are no such pairs. The ideals generated
by $p_v$ and $p_w$ then have zero product, contradicting primeness.
The other maximal-tail axioms
follow from saturated heredity.

Conversely, let $T$ be a maximal tail and let $D=R/\mathfrak p$ be a
domain. In the aperiodic algebra $\operatorname{KP}_D(\Gamma_T)$, every
nonzero ideal contains some $ap_v$ with $0\ne a\in D$; otherwise its
quotient map would be injective by Cuntz--Krieger uniqueness. For two
nonzero ideals choose $ap_v$ and $bp_w$, and use (MT3) to choose a common
descendant $z$. Multiplying by real and ghost paths gives $ap_z$ and
$bp_z$ in the respective ideals. Their product is $abp_z\ne0$.
The quotient is prime. Its surviving vertices and their coefficient
annihilators recover $T$ and $\mathfrak p$, so the parameter map is bijective.

For the topology, for an ideal $J\unlhd A$ write
$D_A(J)=\{P\in\operatorname{Spec}(A):J\not\subseteq P\}$ and put
$D_A(a)=D_A(AaA)$. By the established coefficient-ideal theorem
\cite[Theorem~6.3(a)]{CEHS}, every ideal $J$ is generated by the
elements $rp_v$ which it contains. Consequently
\[
 D_A(J)=\bigcup_{rp_v\in J}D_A(rp_v).
\]
Thus the open sets $D_A(rp_v)$ generate the hull--kernel topology.
Their inverse images are exactly
\[
 \mathcal S(v)\times D_R(r),\qquad
 \mathcal S(v)=\{T:v\in T\}.
\]
These are a basis of the product topology, proving the claim.
\end{proof}

\section{Goldman points and isotropy induction}
\label{sec:goldman-core}

Let $R$ be a nonzero commutative ring with identity. A prime ideal
$\mathfrak p$ of $R$ is called a \emph{Goldman point} if, for
$D=R/\mathfrak p$, there exists $0\ne s\in D$ such that
$D[s^{-1}]$ is a field. This field is necessarily $\operatorname{Frac}(D)$.
Write $\operatorname{GSpec}(R)$ for this subspace of $\operatorname{Spec}(R)$.
Equivalently, Goldman points are the locally closed points of
$\operatorname{Spec}(R)$. Every maximal ideal is a Goldman point, but the
converse fails for a general coefficient ring.

Define $\operatorname{Prim}_{\mathrm{ind}}(A)$ to be the set of annihilators
of modules induced from \emph{simple} modules over isotropy group algebras
of the graph groupoid. This notation does not assert that every primitive
ideal belongs to this set.

\begin{lemma}[A finite-type field detects a Goldman point; see \cite{Stacks}]
\label{lem:goldman-field}
Let $D$ be a domain contained in a field $E$. If $E$ is a finitely generated
$D$-algebra, then $D$ is a Goldman domain.
\end{lemma}
\begin{proof}
Write $F=\operatorname{Frac}(D)$ and $E=D[a_1,\ldots,a_n]$.
Then $E=F[a_1,\ldots,a_n]$, so Zariski's lemma makes $E/F$ finite.
Clear the denominators of the coefficients of monic equations for the
$a_i$ over $F$. For some $0\ne s\in D$, all $a_i$ are integral over
$D[s^{-1}]$. Since $s^{-1}\in E$, we have
$E=D[s^{-1}][a_1,\ldots,a_n]$. Thus a field is integral over
$D[s^{-1}]$, which forces $D[s^{-1}]$ to be a field.
\end{proof}

\begin{lemma}[A dense periodic orbit in every finite maximal tail]
\label{lem:dense-periodic-tail}
For every maximal tail $T$, the groupoid of $\Gamma_T$ has a dense orbit
with nontrivial isotropy. Its isotropy group is free abelian of finite
positive rank.
\end{lemma}
\begin{proof}
Repeated use of (MT3) gives a vertex $w\in T$ reached from every vertex
of $T$. Choose $y\in w\Gamma_T^\infty$ and let
$\mathbf1=(1,\ldots,1)$. Since $T$ is finite, there are $0\leq m<n$
with $y(m\mathbf1)=y(n\mathbf1)=z$. The path
$\lambda=y(m\mathbf1,n\mathbf1)$ is a loop at $z$ with degree
$(n-m)\mathbf1$. Repeating $\lambda$ defines an infinite $\ell$-path
$x=\lambda^\infty$: the degrees of its powers are cofinal in
$\mathbb N^\ell$, so factorisation defines every finite segment.
Every vertex of $T$ reaches $z$.
All paths used here lie in $\Gamma_T$: a path ending in $T$ cannot
pass through its hereditary complement.

For any finite path $\mu$ in $\Gamma_T$, choose a path $\delta$ from
$s(\mu)$ to $z$. Then $\mu\delta x$ belongs to the orbit of $x$ and to
the cylinder $Z(\mu)$. Hence the orbit is dense. Finally
$\sigma^{d(\lambda)}x=x$, so its isotropy is nontrivial. As a subgroup of
$\mathbb Z^\ell$, it is free abelian of finite positive rank.
\end{proof}

\begin{remark}
Aperiodicity does not exclude individual periodic paths. It requires
an aperiodic infinite path at each vertex; the periodic orbit just
constructed is therefore compatible with strong aperiodicity.
\end{remark}

\begin{lemma}[Simple induction over a coefficient ring]
\label{lem:simple-induction-rings}
Let $\mathcal G$ be an ample Hausdorff groupoid, let $R$ be a
commutative unital ring, and let $V$ be a nonzero simple left
$R[G_x]$-module. Then
$RL_x\otimes_{R[G_x]}V$ is a nonzero simple left $A_R(\mathcal G)$-module.
\end{lemma}
\begin{proof}
This is the standard induction theorem; see \cite{Steinberg2010} and,
explicitly over arbitrary coefficient rings, \cite[Theorem~4.2]{DGS}.
We recall its proof to specify the coefficient-ring scope.
For every $y$ in the orbit $\mathcal O_x$, choose an arrow
$\gamma_y:x\to y$, with $\gamma_x=x$. Every arrow in $L_x$ has a
unique form $\gamma_yh$ with $h\in G_x$, so
\[
 RL_x\otimes_{R[G_x]}V
   =\bigoplus_{y\in\mathcal O_x}(\gamma_y\otimes V)
\]
as $R$-modules. In particular it is nonzero. A compact open unit set
$U$ acts as the identity on the summands indexed by $U$ and as zero
on the other summands.

Let $N$ be a nonzero submodule, and choose a nonzero vector of $N$.
Only finitely many summands occur in this vector. A compact open unit
neighbourhood of one of their range points, avoiding the others,
isolates a nonzero element $\gamma_y\otimes v$ of $N$. A compact open
bisection containing $\gamma_y^{-1}$ transports this element to
$x\otimes v$. For every $h\in G_x$, a compact open bisection containing
$h$ acts on the summand at $x$ by $x\otimes v\mapsto x\otimes hv$.
Since $V$ is simple, finite $R$-linear combinations of these operations
show that $x\otimes V\subseteq N$. Bisections containing the arrows
$\gamma_z$ then give $\gamma_z\otimes V\subseteq N$ for every
$z\in\mathcal O_x$. Hence $N$ is the entire induced module.
\end{proof}

\begin{theorem}[The induced primitive locus]
\label{thm:goldman-induced}
Under the standing assumptions,
\[
 \operatorname{Prim}_{\mathrm{ind}}(A)
 =\{P_{T,\mathfrak p}:T\in\operatorname{MT}(\Gamma),
                 \mathfrak p\in\operatorname{GSpec}(R)\}.
\]
In particular, the prime-spectrum homeomorphism restricts to
\[
 \operatorname{MT}(\Gamma)\times\operatorname{GSpec}(R)
 \cong\operatorname{Prim}_{\mathrm{ind}}(A).
\]
\end{theorem}
\begin{proof}
Let $x$ be an infinite path and $V$ a simple module over its isotropy
group algebra $R[G_x]$. The group $G_x$ is a subgroup of
$\mathbb Z^\ell$, hence is free abelian of finite rank. Its group algebra
is commutative and finitely generated over $R$. Thus
$V\cong R[G_x]/\mathfrak m=E$ for a maximal ideal $\mathfrak m$, and $E$
is a field. Set $\mathfrak p=\ker(R\to E)$. Lemma
\ref{lem:goldman-field} implies $\mathfrak p\in\operatorname{GSpec}(R)$.

By Lemma~\ref{lem:simple-induction-rings}, the induced module is simple.
Let $T$ be the set of initial vertices of paths in the orbit of $x$.
For (MT1), prepend a path ending at an orbit vertex to an orbit path.
For (MT2), take the initial segment of any prescribed degree of an
orbit path starting at the given vertex; its shift is still in the
same orbit. For (MT3), suppose orbit paths $y,z$ start at $v,w$.
Choose $p,q,p',q'\in\mathbb N^\ell$ with
$\sigma^p y=\sigma^q x$ and $\sigma^{p'}z=\sigma^{q'}x$.
Both $v$ and $w$ reach $x(q\vee q')$, which is in $T$.
Choose one arrow from $x$ to each orbit point. The induced module is,
as an $R$-module, a direct sum of copies of $E$ indexed by that orbit.
For a vertex $v$, $rp_v$ acts as zero precisely when $v\notin T$ or
$r\in\mathfrak p$. It follows that $P_{T,\mathfrak p}$ annihilates the
module. On the quotient
$\operatorname{KP}_{R/\mathfrak p}(\Gamma_T)$ the coefficient-vertex
test is faithful, so Cuntz--Krieger uniqueness identifies the full
annihilator with $P_{T,\mathfrak p}$.

Conversely, let $T$ be a maximal tail and
$\mathfrak p\in\operatorname{GSpec}(R)$. Put $D=R/\mathfrak p$ and
$F=\operatorname{Frac}(D)=D[s^{-1}]$. By Lemma
\ref{lem:dense-periodic-tail}, choose a path $x$ with dense orbit in
$\Gamma_T$ and with $G_x\cong\mathbb Z^d$, $d\geq1$. Choose a basis
$g_1,\ldots,g_d$ and define a character by
$g_1\mapsto s^{-1}$ and $g_i\mapsto1$ for $i>1$. The resulting map
\[
 R[G_x]\longrightarrow F
\]
is onto. Hence $F$ is a simple isotropy module. Its induced module is
simple by Lemma~\ref{lem:simple-induction-rings}, and the preceding
coefficient-vertex and uniqueness argument
gives annihilator $P_{T,\mathfrak p}$.

Here the orbit computed in the original graph stays in $\Gamma_T$:
every vertex on an orbit path reaches a vertex of $x$, and (MT1)
puts it in $T$. Thus induction in the original graph agrees with the
quotient action used in the proof.

\end{proof}

\begin{corollary}[A topological characterisation of the induced locus]
\label{cor:induced-locally-closed}
The annihilators induced from simple isotropy modules are exactly the
locally closed points of $\operatorname{Spec}(A)$.
\end{corollary}
\begin{proof}
For a prime $\mathfrak p$ of $R$, the point $\mathfrak p$ is locally
closed precisely when there exists $s\notin\mathfrak p$ such that
$D_R(s)\cap V_R(\mathfrak p)=\{\mathfrak p\}$. Equivalently the domain
$(R/\mathfrak p)[\bar s^{-1}]$ has only its zero prime and is therefore
a field. This is exactly the Goldman condition.

The space $\operatorname{MT}(\Gamma)$ is finite and $T_0$: distinct
tails are distinguished by a vertex. Every point of a finite $T_0$
space is locally closed. A point $(T,\mathfrak p)$ of its product with
$\operatorname{Spec}(R)$ is locally closed if $\mathfrak p$ is locally
closed, by taking products of locally closed singletons. Conversely,
intersecting a locally closed singleton in the product with the slice
$\{T\}\times\operatorname{Spec}(R)$ makes $\mathfrak p$ locally closed
in that slice. Apply Proposition~\ref{prop:all-prime-coefficients} and
Theorem~\ref{thm:goldman-induced}.
\end{proof}

\section{Primitive algebras over nonfield coefficient rings}
\label{sec:nonfield-examples}

\begin{corollary}[Nonfield coefficient rings can give primitive algebras]
\label{cor:dvr-primitive}
Let $R$ be a Goldman domain and let $n\geq2$. Then
$L_R(1,n)$ is primitive. In particular this holds for the nonfield
$R=k[t]_{(t)}$, where $k$ is any field. If $R$ is a discrete valuation
ring, every prime ideal of $\operatorname{KP}_R(\Gamma)$ is primitive,
and
\[
 \operatorname{Prim}\operatorname{KP}_R(\Gamma)
 \cong\operatorname{MT}(\Gamma)\times\operatorname{Spec}(R).
\]
\end{corollary}
\begin{proof}
The rose with $n\geq2$ loops is cofinal and strongly aperiodic, so the
zero ideal is supplied by Theorem~\ref{thm:goldman-induced}. For a DVR,
both prime ideals, zero and the maximal ideal, are Goldman points.
Proposition~\ref{prop:all-prime-coefficients} and the theorem therefore
account for every prime ideal, and every primitive ideal is prime.
\end{proof}

\begin{proposition}[An explicit faithful simple module]
\label{prop:explicit-goldman-module}
Let $R$ be a domain and let $0\ne s\in R$ satisfy
$F=R[s^{-1}]=\operatorname{Frac}(R)$. Let $W$ consist of the empty
word $\epsilon$ and the finite words on $\{a,b\}$ ending in $b$.
The following operators make
\[
 V=\bigoplus_{w\in W}F e_w
\]
a faithful simple module over $L_R(1,2)$. The operators below are
extended by $F$-linearity:
\[
 ae_\epsilon=s^{-1}e_\epsilon,\quad
 a^*e_\epsilon=s e_\epsilon,\quad
 be_\epsilon=e_b,\quad b^*e_\epsilon=0.
\]
For nonempty $w\in W$, set
\begin{gather*}
 ae_w=e_{aw},\qquad be_w=e_{bw},\\
 a^*e_w=\begin{cases}e_u,&w=au,\\0,&w=bu,\end{cases}
 \qquad
 b^*e_w=\begin{cases}e_u,&w=bu,\\0,&w=au.\end{cases}
\end{gather*}
The vertex acts as the identity and coefficients in $R$ act by scalar
multiplication through $R\subseteq F$.
\end{proposition}
\begin{proof}
Deleting the first letter of a nonempty canonical word leaves a canonical
word or $\epsilon$, so the formulas are well defined. At $e_\epsilon$,
the operators $a$ and $a^*$ have inverse weights $s^{-1}$ and $s$;
on the other basis vectors, adding and removing the same first letter
are inverse operations. Adding one letter and removing the other gives
zero, including at $e_\epsilon$. Finally $aa^*$ fixes $e_\epsilon$
and the words beginning with $a$, while $bb^*$ fixes precisely the words
beginning with $b$. Therefore
\[
 a^*a=b^*b=1,\qquad a^*b=b^*a=0,\qquad aa^*+bb^*=1.
\]
They therefore define an $L_R(1,2)$-module.

To prove simplicity explicitly, let
$\xi=\sum_{w\in S}c_we_w\ne0$ with $S\subset W$ finite and all
$c_w\ne0$. Choose $w_0\in S$ and interpret $\epsilon^*=1$.
For a finite word $w=c_1\cdots c_m$, write
$w^*=c_m^*\cdots c_1^*$. Applying $w_0^*$ deletes the prefix $w_0$
whenever it is present. This gives
\[
 w_0^*\xi=c_{w_0}e_\epsilon+
                   \sum_{u\in F_0}d_u e_u,
 \qquad F_0\subset W\setminus\{\epsilon\}\text{ finite}.
\]
Indeed no other canonical word can be transported to $e_\epsilon$:
canonical nonempty words end in $b$, so any remaining nonempty suffix
cannot consist only of $a$'s. If a shorter canonical word is exhausted
while deleting $w_0$, the remaining suffix of $w_0$ contains its final
$b$, whose ghost kills $e_\epsilon$.
Choose $N$ larger than the initial run of $a$'s in every $u\in F_0$ and
put $Q_N=a^N(a^*)^N$. Then
\[
 Q_Ne_\epsilon=e_\epsilon,\qquad Q_Ne_u=0\ (u\in F_0),\qquad
 Q_Nw_0^*\xi=c_{w_0}e_\epsilon.
\]
Write $c=c_{w_0}\ne0$. Since $c^{-1}\in F=R[s^{-1}]$, there are
$r\in R$ and $m\geq0$ such that $c^{-1}=rs^{-m}$. Hence
\[
                     (ra^m)(ce_\epsilon)=e_\epsilon.
\]
The same scalar operations give every element of $Fe_\epsilon$.
Prepending a canonical word $w$ then gives $Fe_w$. Thus the submodule
contains all of $V$, without assuming that an $L_R(1,2)$-submodule
is initially an $F$-subspace.
Thus $V$ is simple. Finally every $0\ne r\in R$ acts nontrivially at
the sole vertex. The Cuntz--Krieger uniqueness theorem
\cite[Theorem~6.5]{Tomforde}, since every
cycle in the two-loop rose has an exit, makes the representation
faithful. This proves the proposition.
\end{proof}

\begin{remark}[Left ideals and annihilators]
\label{rem:left-ideal-annihilator}
Proposition~\ref{prop:explicit-goldman-module}, with $R=k[t]_{(t)}$
and $s=t$, gives a counterexample to the coefficient-field necessity
assertions in \cite[Theorem~6.7, arXiv version~2]{Larki} and
\cite[Theorem~3.7, arXiv version~2]{KLA}.
In their necessity arguments, membership in the left ideal presenting
a simple module is used as if it implied annihilation of that module.
For a unital ring $B$ and a maximal left ideal $J$, however,
\[
 \Ann_B(B/J)=\{b\in B:bB\subseteq J\}\subseteq J,
\]
and the inclusion can be strict. Thus $b\in J$ only guarantees
$b(1+J)=0$; it need not give $b(B/J)=0$.
In the module above, $b^*e_\epsilon=0$ but $b^*e_b=e_\epsilon$,
which exhibits this distinction directly.
The counterexample and this comparison concern the specified arXiv
versions; no assertion is made here about unexamined versions.
\end{remark}

\begin{remark}[Jacobson coefficients and the full primitive spectrum]
\label{rem:jacobson-boundary}
If $R$ is Jacobson, then $\operatorname{GSpec}(R)=\operatorname{MaxSpec}(R)$:
for a Goldman prime $\mathfrak p$, the field
$(R/\mathfrak p)[s^{-1}]$ is a finite-type algebra over the Jacobson
domain $R/\mathfrak p$, forcing that domain to be a field
\cite[Section~10.35]{Stacks}. Thus the induced primitive locus in
Theorem~\ref{thm:goldman-induced} is
$\operatorname{MT}(\Gamma)\times\operatorname{MaxSpec}(R)$.
The DVR example concerns a non-Jacobson coefficient ring.

The equality in Corollary~\ref{cor:induced-locally-closed} concerns
annihilators induced from simple isotropy modules. General induction
theorems do not identify this set with the full primitive spectrum
without an additional argument; see \cite{SteinbergEH}. No such
identification is assumed here.
\end{remark}

\begin{corollary}[Coefficient rings for which every prime is primitive]
\label{cor:semilocal}
Let $R$ be a nonzero commutative unital ring and let $\Gamma$
satisfy the hypotheses of the main theorem. If either
\begin{enumerate}[label=\textup{(\roman*)},leftmargin=2em]
\item $\Spec R$ is finite, or
\item $R$ is semilocal and $\dim R\leq1$,
\end{enumerate}
then
\[
 \Prim\KP_R(\Gamma)=\Spec\KP_R(\Gamma)
 \cong\MT(\Gamma)\times\Spec R.
\]
\end{corollary}
\begin{proof}
It suffices to show $\GSpec R=\Spec R$ and apply
Theorem~\ref{thm:goldman-induced}. In case \textup{(i)}, every point
of the finite $T_0$ space $\Spec R$ is locally closed.

In case \textup{(ii)}, fix $\mathfrak p\in\Spec R$. The quotient
$D=R/\mathfrak p$ is a semilocal domain of dimension at most one.
If $D$ is a field, there is nothing to prove. Otherwise, write its
finitely many maximal ideals as $\mathfrak m_1,\ldots,\mathfrak m_q$,
choose $0\ne s_i\in\mathfrak m_i$, and put $s=\prod_i s_i\ne0$.
Every nonzero prime of $D$ is maximal and contains $s$. Thus
$D[s^{-1}]$ has only its zero prime and is a field. Hence
$\mathfrak p$ is Goldman. Neither reducedness nor Noetherianity of
$R$ is needed.
\end{proof}

Recall that an ideal is \emph{completely irreducible} if it cannot be
expressed as an intersection of ideals each strictly larger than itself.

\begin{example}[Locally closed does not imply completely irreducible]
\label{ex:not-completely-irreducible}
Let $R=k[t]_{(t)}$ and $A=L_R(1,2)$. Its prime spectrum is
$\{0,tA\}$ by Proposition~\ref{prop:all-prime-coefficients}, and
both primes are primitive. In particular, $D_A(t)=\{0\}$, so zero
is locally closed. Nevertheless,
\[
                 \bigcap_{n\geq1}t^nA=0,
                 \qquad t^nA\ne0\quad(n\geq1).
\]
Indeed, the universal relations give
$A\cong R\otimes_k L_k(1,2)$. A $k$-basis of $L_k(1,2)$ makes this
a free $R$-module. Each element has finite support in that basis, and
$\bigcap_n t^nR=0$, which proves the intersection assertion.
The ideals $t^nA$ are nonzero since scalar vertices are nonzero.
Thus zero is an intersection of ideals properly containing it; it is
not completely irreducible. This contrasts with
\cite[Theorem~3.9]{Rangaswamy}, whose coefficients are a field.
\end{example}

\begin{example}[A four-point primitive spectrum]
Let $R$ be a discrete valuation ring with maximal ideal $(t)$.
Take a directed graph with vertices $v,w$, two loops at each vertex,
and one edge from $v$ to $w$. Denote this graph by $E$ and put
$A=\KP_R(E)$. Its maximal tails are $\{v,w\}$ and
$\{v\}$; the two loops ensure strong aperiodicity. Its prime and
primitive ideals coincide and are
\[
                     0,\quad tA,\quad I_{\{w\}},\quad I_{\{w\}}+tA.
\]
The middle two ideals are incomparable, and their sum is the last
ideal. Thus the specialization order is the product of two two-point
chains, as predicted by the spectrum homeomorphism.
\end{example}

\section{Split pullbacks and central coefficient fibres}
\label{sec:pullbacks}
We apply the standard groupoid decomposition for pullbacks to the
Goldman description of the induced locus. The exhaustion result at the end
of this section combines that description with the usual
countable-dimensional Dixmier argument.

Let $f:\N^k\twoheadrightarrow\N^\ell$ be a surjective monoid
homomorphism, with group completion
$\bar f:\Z^k\to\Z^\ell$ and kernel $H$. Put $C=R[H]$.
The pullback graph has paths
\[
 f^*\Gamma=\{(\lambda,n):d(\lambda)=f(n)\},
 \qquad d(\lambda,n)=n.
\]
Choose a group homomorphism $j:\Z^\ell\to\Z^k$ with
$\bar f j=\mathrm{id}$, and write $c(n)=n-j(\bar f(n))\in H$.
Such a splitting exists because $\Z^\ell$ is free.

\begin{proposition}[The standard split pullback identification]
\label{prop:pullback-split}
There is an $R$-algebra isomorphism
\begin{equation}\label{eq:pullback-split}
 \KP_R(f^*\Gamma)\cong\KP_R(\Gamma)\otimes_R R[H]
 \cong\KP_C(\Gamma).
\end{equation}
It sends $p_v$ to $p_v\otimes1$ and
$s_{(\lambda,n)}$ to $s_\lambda\otimes u^{c(n)}$,
where $u^h$ is the basis element of $R[H]$ associated to $h$.
The corresponding ghost generator is sent to
$s_{\lambda^*}\otimes u^{-c(n)}$.
The corresponding graph-groupoid isomorphism is
$\mathcal G_{f^*\Gamma}\cong\mathcal G_\Gamma\times H$.
\end{proposition}
\begin{proof}
The existence of the product-groupoid description is standard
\cite[Proposition~2.10]{KP2000}; together with tensor and coefficient
extension for Steinberg algebras \cite{RigbyTensor} it gives
\eqref{eq:pullback-split}. We record the concrete formulas.

Every infinite path $x$ of $\Gamma$ has a lift
$\widetilde x(m,n)=(x(f(m),f(n)),n-m)$.
To describe the inverse, choose a monoid section
$\tau:\N^\ell\to\N^k$ of $f$. Such a section exists: expressing
each standard basis vector of $\N^\ell$ as a nonnegative integral
combination of the vectors $f(e_i)$ shows that some $f(e_i)$ equals
that basis vector. Choosing one such $e_i$ for each coordinate defines
$\tau$. This auxiliary section need not agree with the chosen group
section $j$.

For an infinite path $y$ of $f^*\Gamma$, let $\pi_\Gamma$ forget
the pullback degree and put
$x(p,q)=\pi_\Gamma(y(\tau(p),\tau(q)))$.
For $m\leq n$ in $\N^k$, choose $L\geq n,\tau(f(n))$.
Factoring the projected finite path $\pi_\Gamma(y(0,L))$ at
degrees $f(m),f(n)$ in the two ways supplied by $m,n$ and by
$\tau(f(m)),\tau(f(n))$ gives
$\pi_\Gamma(y(m,n))=x(f(m),f(n))$. Hence $y=\widetilde x$.
Since $f\tau=\mathrm{id}$, the two constructions are inverse.
The lift carries the base cylinder $Z(\lambda)$ onto the pullback
cylinder $Z((\lambda,a))$ for any $a$ with $f(a)=d(\lambda)$,
so it is a homeomorphism of infinite-path spaces. It satisfies
$\sigma^a\widetilde x=\widetilde{\sigma^{f(a)}x}$.
On arrows the isomorphism is
\[
 (\widetilde x,n,\widetilde y)
       \longmapsto ((x,\bar f(n),y),c(n)).
\]
For surjectivity on arrows, start with $(x,p-q,y)$ and $h\in H$,
choose $a,b\in\N^k$ lifting $p,q$, and set
$h'=h-c(a-b)$. Write $h'=u-v$ with $u,v\in\N^k$.
Then $f(u)=f(v)$, so shifting both equal tails by this common degree
shows that $(\widetilde x,a+u-b-v,\widetilde y)$ is an arrow
with the required image. Injectivity follows from the decomposition
$\Z^k=j(\Z^\ell)\oplus H$. More explicitly, a pullback cylinder
$Z((\alpha,a),(\beta,b))$ is carried onto
$Z(\alpha,\beta)\times\{c(a-b)\}$. These images also refine every
basic product bisection: for $Z(\alpha,\beta)\times\{h\}$, make
the preceding adjustment $u,v$ and extend both $\alpha$ and $\beta$
by a common path $\delta$ of degree $f(u)=f(v)$. The resulting
pullback cylinders have images
$Z(\alpha\delta,\beta\delta)\times\{h\}$, which cover the basic
product bisection. Thus the map is open and continuous. The displayed
generator formulas follow by taking characteristic functions of these
bisections.
\end{proof}

\begin{corollary}[Prime and simple-induced spectra of pullbacks]
\label{cor:pullback-spectra}
Under the hypotheses of the main theorem on $\Gamma$, let
$B=\KP_R(f^*\Gamma)$. Then
\[
 \Spec B\cong\MT(\Gamma)\times\Spec C,\qquad
 \Prim_{\mathrm{ind}}B\cong\MT(\Gamma)\times\GSpec C.
\]
The second space is precisely the locally closed locus in $\Spec B$.
Here induction is taken in $\mathcal G_{f^*\Gamma}$ over $R$.
\end{corollary}
\begin{proof}
Apply Proposition~\ref{prop:all-prime-coefficients} and
Theorem~\ref{thm:goldman-induced} to $\KP_C(\Gamma)$.
For the assertion about induction, the product-groupoid identification
gives $R[H\oplus G_x]\cong C[G_x]$ at every unit $x$.
The source fibre is $L_x\times H$. For a module $V$ over
$R[H\oplus G_x]=C[G_x]$, the identification of induced modules is
\begin{align*}
 R(L_x\times H)\otimes_{R[H\oplus G_x]}V
 &\longrightarrow CL_x\otimes_{C[G_x]}V,\\
 (\gamma,h)\otimes v&\longmapsto\gamma\otimes u^h v.
\end{align*}
The balancing relation follows from the commuting actions of $H$ and
$G_x$, and the inverse sends $\gamma\otimes v$ to
$(\gamma,0)\otimes v$. These maps intertwine convolution under
\eqref{eq:pullback-split}. Thus the transported representations are
exactly those induced from simple isotropy modules of the pullback
groupoid.
\end{proof}

\begin{theorem}[Full primitive spectrum over an uncountable field]
\label{thm:uncountable-field}
Let $K$ be any uncountable field and $C=K[H]$. For $\Gamma$ as in the
main theorem,
\[
 \Prim\KP_K(f^*\Gamma)
  =\Prim_{\mathrm{ind}}\KP_K(f^*\Gamma)
  \cong\MT(\Gamma)\times\MaxSpec C.
\]
\end{theorem}
\begin{proof}
Use Proposition~\ref{prop:pullback-split} to write the algebra as
$B=\KP_C(\Gamma)$. This is countable-dimensional over $K$:
the graph is countable and $H$ is finitely generated. Every simple
$B$-module $M$ is cyclic and hence countable-dimensional over $K$.

We first recall the countable-dimensional form of the usual
Dixmier argument. Every $\varphi\in\operatorname{End}_B(M)$ is
algebraic over $K$. Otherwise all $\varphi-a$, $a\in K$, would be
invertible by Schur's lemma. For a fixed $0\ne m\in M$, the vectors
$(\varphi-a)^{-1}m$ would be linearly independent. Indeed, a finite
relation multiplied by $\prod_i(\varphi-a_i)$ gives $q(\varphi)m=0$,
where $q(X)=\sum_i c_i\prod_{j\ne i}(X-a_j)$.
If any $c_i\ne0$, evaluating at $a_i$ shows that $q\ne0$.
Since $m$ generates $M$, $q(\varphi)m=0$ implies $q(\varphi)=0$,
a contradiction. An uncountable independent family cannot lie in $M$.

The central action of $C$ lies in the division ring
$\operatorname{End}_B(M)$. Its image is a commutative finitely
generated $K$-algebra generated by algebraic elements, so it is a
finite-dimensional domain over $K$, hence a field. Consequently,
$\mathfrak m=\Ann_B(M)\cap C$ is maximal. The prime parametrisation
forces $\Ann_B(M)=P_{T,\mathfrak m}$ for some $T$.
Conversely every such pair is primitive by
Theorem~\ref{thm:goldman-induced}. Restricting the prime-spectrum
homeomorphism proves the topological assertion.
\end{proof}

\begin{remark}
Since $H\cong\Z^{k-\ell}$, the ring $C$ is a Laurent polynomial
algebra over $K$ and is Jacobson. Thus
$\GSpec C=\MaxSpec C$ over any field. Uncountability is used above
only for the exhaustion of the full primitive spectrum. Algebraic
closedness is not required. If $K$ is algebraically closed, the
Nullstellensatz identifies $\MaxSpec C$ with the algebraic torus
$(K^\times)^{k-\ell}$ after a basis of $H$ is chosen.
\end{remark}

\section{Further questions}
The main theorem leaves open, within this argument, whether every
primitive ideal of $\KP_R(\Gamma)$ is induced from a simple isotropy
module for an arbitrary commutative ring $R$. The result gives a
precise formulation: this exhaustivity is equivalent to saying that
every primitive coefficient contraction is a Goldman prime.
The semilocal result and Theorem~\ref{thm:uncountable-field} settle
particular classes, but do not establish the assertion in general.

The finite-vertex assumption has two separate uses. It makes every
maximal-tail point locally closed, and it supplies a common accessible
periodic path in each tail. For infinite graphs, these properties can
fail independently. An extension would have to account for the
topology of the maximal-tail space and the existence of enough
isotropy to realise nonmaximal Goldman primes.

Recent work of Dokuchaev, Exel and Pinedo on twisted Steinberg
algebras and induction \cite[Theorem~12.14 and Question~12.15]{DEP}
also separates induction of primitive ideals from primitive isotropy data. Our proof uses commutative, finite-type
isotropy algebras and the uniqueness theorem in an essential way;
it therefore does not directly extend to arbitrary twisted isotropy.

\section*{Statements and Declarations}
\paragraph{Funding.}
This research is partially supported by the Institute of Mathematics, VAST under grant number CSCL01.01/25-26.
\paragraph{Competing interests.}
The author declares that she has no competing interests.
\paragraph{Author contributions.}
Nguyen Bich Van is the sole author listed in this manuscript.
\paragraph{Data availability.}
This is a theoretical mathematical study and uses no empirical data.
The proofs are contained in the article.

\end{document}